\documentclass[twocolumn]{autart} 
\usepackage{amsmath,amsfonts,amssymb}
 \usepackage[utf8]{inputenc}     
\usepackage{nicematrix}
\usepackage{tikz,mathrsfs}
\usepackage{commath}
\usepackage{comment}
\usetikzlibrary{fit}
\usepackage{algorithm}
\usepackage{algorithmicx,algcompatible}
\usepackage[compatible]{algpseudocode}
\usepackage{float,subfigure}

\DeclareMathOperator{\rank}{rank}
\DeclareMathOperator{\blkdiag}{blk-diag}

\DeclareMathOperator{\oC}{\overline{\mathbb{C}^+}}
\DeclareMathOperator{\R}{\mathbb{R}}

\newcommand{\C}{{\mathbb{C}}}

\makeatletter
\newenvironment{breakablealgorithm}
{
		\begin{center}
			\refstepcounter{algorithm}
			\hrule height.8pt depth0pt \kern2pt
			\renewcommand{\caption}[2][\relax]{
{\raggedright\textbf{\fname@algorithm~\thealgorithm} ##2\par}%
				\ifx\relax##1\relax 
				\addcontentsline{loa}{algorithm}{\protect\numberline{\thealgorithm}##2}%
				\else 
				\addcontentsline{loa}{algorithm}{\protect\numberline{\thealgorithm}##1}%
				\fi
				\kern2pt\hrule\kern2pt
			}
		}{
		\kern2pt\hrule\relax
	\end{center}
}
\makeatother

\begin{document}
	
	\begin{frontmatter}
		
		\title{Distributed partial state estimation for linear state-space systems\thanksref{footnoteinfo}} 
		
		\thanks[footnoteinfo]{This paper was not presented at any IFAC meeting. Corresponding author: Nutan Kumar Tomar.	}
		
		\author[inst1]{Juhi Jaiswal}\ead{juhi$\_$jaiswal@iiitvadodara.ac.in},	
		\author[inst2]{Thomas Berger}\ead{thomas.berger@mathematik.uni-halle.de},    
		\author[inst3]{Nutan Kumar Tomar}\ead{nktomar@iitp.ac.in}              
		
		\address[inst1]{Indian Institute of Information Technology Vadodara, Gandhinagar - 382028, Gujarat, India} 	
		
\address[inst2]{Institut f\"ur Mathematik, Martin-Luther-Universit\"at Halle-Wittenberg, Theodor-Lieser-Stra\ss e 5, 06120 Halle, Germany}

		\address[inst3]{Indian Institute of Technology Patna, Patna - 801103, Bihar, India}

		
		\begin{abstract}          
This study is concerned with the problem of partial state estimation for linear time-invariant (LTI) distributed state-space systems. A necessary and sufficient condition is established in terms of a simple rank criterion involving the system coefficient matrices, provided the communication graph is either directed, balanced and strongly connected or undirected and connected. The estimator parameter matrices are obtained by simple matrix theory. Finally, a numerical example demonstrates the feasibility and effectiveness of the proposed theoretical results and design algorithm.
		\end{abstract}
	\end{frontmatter}

\section{Introduction} 
In recent years, the rapid advancement in sensor technology,  affordable computational resources, and multi-agent cooperation technology, have paved the way for employing large-scale and spatially deployed, intricate control systems in numerous technical and societal domains. A large portion of modern control theory bases its designs on the assumption that the state vector of the system to be controlled is available for measurement. In many real-world scenarios, only a few output quantities are available. The application of theories that assume availability of state measurements is severely limited in these cases \cite{luenberger1964observing}. An early contribution to construct an estimate of the system state  via the available system inputs and
outputs was made by Luenberger \cite{luenberger1964observing}. The device which reconstructs the state vector is called an estimator (observer). However, in large-scale systems, no sufficient number of measurements is available at a single location to produce an asymptotic estimate of the full plant state. Thus, the problem of distributed state estimation becomes an emerging research topic in modern control theory. The challenge is to design a distributed estimator, i.e., a bunch of local estimators connected via a network topology. Each local estimator individually estimates the state vector via its own limited measurements (input and output) and the communication received from neighboring local estimators. A network of estimators is said to achieve omniscience asymptotically, if all of their states converge to the plant state \cite{zhu2014cooperative}. 

Over the past two decades, a significant development has been achieved for the design of distributed full-state estimators for linear time-invariant (LTI) distributed state-space systems. For instance, Park and Martins~\cite{park2012augmented} provided a sufficient condition for the design of a network of augmented observers for the distributed estimation problem in terms of the eigenvalues of the Laplacian matrix of the underlying communication graph and the spectral radius of the dynamic matrix of the LTI system. Further,~\cite{park2012necessary} generalized the results of~\cite{park2012augmented} and established necessary and sufficient conditions for the existence of omniscience-achieving design parameters for the augmented distributed observer in terms of the weight matrix, the system matrix, and observation matrix. Ugrinovskii~\cite{ugrinovskii2013conditions} presented necessary and sufficient conditions for the detectability of a linear plant via a network of state estimators. Further, Zhu et al.~\cite{zhu2014cooperative}  established a necessary and sufficient condition for the existence of observer parameters that ensure that all observers over an undirected connected network achieve omniscience asymptotically under the assumption of joint detectability. Also,~\cite{doostmohammadian2013genericity} explored agent communication's role in error stability in distributed estimation, demonstrating its ability to track unstable systems. Mitra and Sundaram~\cite{mitra2016approach} provided necessary and sufficient conditions for a strongly connected graph so that the proposed observer can achieve omniscience. Kim et al.~\cite{kim2016distributed} studied distributed Luenberger observer design for an undirected and connected communication graph under joint detectability of the state-space systems. Moreover, Park and Martins~\cite{park2017design} developed necessary and sufficient conditions in terms of the detectability of the subsystems of the plant that are associated with the source components of the graph for the existence of a parameter choice for a distributed observer that achieves omniscience asymptotically and satisfies the scalability condition. Han et al.~\cite{han2018simple} presented a simple algorithm to design a distributed observer. Yang et al.~\cite{yang2022state} proposed a distributed approach that estimates the system's state vector by each observer, if the union of communication links in bounded intervals makes the network communication graph connected. Cao and Wang~\cite{cao2023distributed} proposed a novel distributed full-state observer design approach for linear systems with unknown inputs under the assumption of a strongly connected graph and a version of detectability. Liang et al.~\cite{liang2024distributed}  presented a novel set of distributed unknown input observers in the context of the distributed state and fault estimation of cyber-physical systems under DoS attacks. Zhou and Zhao~\cite{zhou2025adaptive} presented an adaptive distributed unknown input observer under a directed topology based on the global output estimation method. 

It is noteworthy here that all the above mentioned works focus on distributed full-state estimation. 
However, in many real-world applications such as feedback control, fault diagnosis, process monitoring, etc., only the information of some part or linear combination of the state vector (components) is required. The problem of estimating this particular part without estimating the full-state vector of a networked or spatially deployed system is known as distributed partial state estimation. Remarkably, partial state estimators can be designed under much less restrictive conditions and computational loads than those required for full-state estimators. As a result,  research on partial state estimator design is still ongoing, even for centralized state-space systems~\cite{darouach2025functional}.  To the best of our knowledge, the problem of distributed partial state estimation has not been addressed  in the literature yet.

Motivated by these facts, the present study aims to solve the problem of distributed partial state estimation for LTI distributed state-space systems of the form 
\begin{subequations}\label{eqn1}
	\begin{align}
\dot{x}(t) &= Ax(t) + Bu(t), \label{eqn1a} \\
y_i(t) &= C_ix(t) +D_iu(t) ,\qquad i = 1,2 \ldots,l, \label{eqn1b} \\
z(t) &= Kx(t), \label{eqn1c}
	\end{align}
\end{subequations}
where $x: \mathbb{R} \to \mathbb{R}^n$, $u: \mathbb{R} \to \mathbb{R}^m$, $y_i: \mathbb{R} \to \mathbb{R}^{p_i}$, and  $z: \mathbb{R} \to \mathbb{R}^r$ are the state vector, the input vector, the output vector, and the functional vector (to be estimated), respectively. $A \in \mathbb{R}^{n \times n}$, $B \in \mathbb{R}^{n \times m}$, $C_i \in \mathbb{R}^{p_i \times n}$, $D_i \in \mathbb{R}^{p_i \times m}$, for $i \in \{1,2,\ldots,l\}$, and $K \in \mathbb{R}^{r \times n}$ are known constant matrices.

Addressing this problem requires a systematic approach that begins with centralized partial state estimator design for LTI state-space systems of the form
\allowdisplaybreaks
\begin{subequations}\label{dls}
\begin{align}
\dot{x}(t) &= Ax(t) + Bu(t), \label{dlsa} \\
y(t) &= Cx(t) + Du(t), \label{dlsb} \\
z(t) &= Kx(t), \label{dlsc}
\end{align}
\end{subequations} 
before
moving to distributed partial state estimation. In~\eqref{dlsb}, we additionally have $y: \mathbb{R} \to \mathbb{R}^{p}$, $C \in \mathbb{R}^{p \times n}$, and $D \in \mathbb{R}^{p\times m}$. This paper presents the following key contributions to advance the state estimation theory for distributed systems:
\begin{enumerate}
\item A novel partial state estimator design is proposed for partially detectable (centralized) state-space system \eqref{dls}. This approach provides a foundational ground for the distributed partial state estimator design. 

\item A necessary and sufficient condition for the existence of distributed partial state estimators for system~\eqref{eqn1} is established. Notably, this condition  is exactly the joint partial detectability of~\eqref{eqn1}. Hence, no additional assumptions are required on the system and/or the communication graph.

\item A numerical algorithm is proposed for the design of distributed partial state estimators. 
\end{enumerate}

This paper is organized as follows. In Section \ref{sec:pre}, we present some preliminaries required for the development of the main results.
Section \ref{sec:classical} provides a novel method for the design of partial state estimators for systems~\eqref{dls}. Necessary and sufficient conditions  for the existence of distributed partial state estimators for systems~\eqref{eqn1} are given in Section~\ref{sec:distributed}. An illustrative example is presented in Section~\ref{sec:numerical} to validate the theoretical findings. Finally, some conclusions are given in Section~\ref{sec:conclusion}.

Throughout the article, we use the symbol $\otimes$ to represent the Kronecker product. $0$ and $I$ denote zero  and identity matrices of appropriate dimensions, resp. The set of complex numbers is denoted by $\mathbb{C}$, $\mathbb{C}^- := \{\lambda \in \mathbb{C}~| ~ \text{Re}(\lambda)< 0\}$, and $\oC := \{\lambda \in \mathbb{C}~| ~ \text{Re}(\lambda) \geq 0\} $. The notation ``$z(t) \to 0$ as $t \to \infty$" means ``$\lim_{t \to \infty} {\rm ess{-}sup}_{s\ge t} \norm{z(s)} = 0$". In a block matrix $\begin{bmatrix}
    A_1 & A_2 \\ * & A_3
\end{bmatrix}$, an asterisk ``$*$" represents $A_2^\top$.
For matrices $A$ and $C$ of appropriate dimensions, 
$\mathcal{D}_{ [A,C],n,\lambda} := \begin{bmatrix}
C \\ C(\lambda I_n-A) \\ \vdots \\ C(\lambda I_n-A)^{n-1} \\ (\lambda I_n-A)^{n} \end{bmatrix}.$

\section{Preliminaries}\label{sec:pre}
In this section, we describe the notations used in this paper and recall some basic results from systems theory, graph theory, and linear algebra, used in the development of main results of this paper.

 \begin{prop}[{\cite[Thm $2.14$]{hardy2019matrix}}]\label{prop:eigen:2}
 Let $A\in\C^{m \times m}$ with eigenvalue $\lambda\in \sigma(A) := \{ \lambda \in \C \mid \det(\lambda I -A)=0 \}$ and corresponding eigenvector $u\in\C^m$. Let $B\in\C^{n \times n}$ with eigenvalue $\mu\in\sigma(B)$ and corresponding eigenvector $v\in\C^n$. Then $\lambda \mu\in \sigma(A \otimes B)$  with corresponding eigenvector $u \otimes v$. 
 \end{prop}

Now, we recall Ger{\v{s}}gorin's result, which is used to bound the spectrum of a square matrix.
\begin{prop}[{\cite[Thm. $1.1$]{varga2011gervsgorin}}]\label{prop:gervsgorin:thm}
For any $A = [a_{ij}] \in \mathbb{C}^{n \times n}$ and $\lambda \in \sigma(A)$, there is an index $k\in\{1,\ldots,n\}$ such that
$\abs{ \lambda - a_{kk} }\leq \sum_{i = 1,\, i \neq k}^{n} \abs{a_{ki}} =: r_k(A)$. 
Consequently, $\lambda \in \Gamma_k(A) := \{z \in \C \mid \abs{z - a_{kk}} \leq r_k(A)\}$, and hence, $\lambda \in \Gamma(A) := \bigcup_{k = 1}^{n} \Gamma_k(A)$. Therefore,  $\sigma(A) \subseteq \Gamma(A)$. 
\end{prop}

The following proposition is a direct consequence of the real Schur decomposition~\cite{piziak2007matrix} and the application of permutation matrices.
\begin{prop}\label{lm:lower:triangular}
For any matrix $A \in \mathbb{R}^{n \times n}$, there exists an orthogonal matrix $P \in \mathbb{R}^{n \times n}$ such that $P^\top AP = \begin{bNiceMatrix}
R_{11} & 0 &  \Ldots & 0 \\ 
\Vdots & \Ddots & \Ddots & \Vdots \\
& & & 0 \\ R_{m1} & \Ldots & & R_{mm}
\end{bNiceMatrix}, $
where each $R_{ii}$ is either a $1 \times 1$ matrix having real eigenvalues or a $2 \times 2$ matrix having complex conjugate eigenvalues. 
\end{prop}

\begin{prop}[{\cite{piziak2007matrix}}]\label{pre:prop5}
Let $X \in \mathbb{R}^{m_1 \times r_1}$ and $Y \in \mathbb{R}^{r_1 \times n_1}$. If $\rank X = r_1$, then $\rank(XY) = \rank Y$.
\end{prop}

\begin{prop}[{\cite{matsaglia1974equalities}}]\label{pre:prop1}
Let $X \in \mathbb{R}^{m_1 \times r_1},~S \in \mathbb{R}^{m_2 \times r_1}$, and $Y \in \mathbb{R}^{m_2 \times r_2}$. If $\rank X = r_1$ or $\rank Y = m_2$, then
$\rank \begin{bmatrix}
X & 0 \\ S & Y
\end{bmatrix} = \rank{X} + \rank{Y}.$
\end{prop}

We conclude this section by introducing the notion of a communication graph.
Communication among the individual subsystems of~\eqref{eqn1} is described by an unweighted graph $\mathcal{G} = (\mathbf{N}, \mathbf{E}, \mathcal{A})$, where $\mathbf{N} = \{1,2,\ldots,l\}$ is the set of nodes,
$\mathbf{E} \subseteq \mathbf{N} \times \mathbf{N}$ is the set of edges, and $\mathcal{A} = [\gamma_{ij}] \in \R^{l \times l}$ denotes the adjacency matrix. Similar to \cite{yang2022state}, we say that $$\gamma_{ij} = \begin{cases}
    1,~~\text{if node $j$ sends information to node $i$,} \\
    0,~~\text{otherwise}.
\end{cases} .$$
Moreover, the Laplacian matrix associated with the communication graph $\mathcal{G}$ is denoted by $\mathcal{L} = \begin{bmatrix}
	\alpha_{ij} 
\end{bmatrix} \in \mathbb{R}^{l \times l}$, where \begin{equation*}
 \alpha_{ij} = \begin{cases}
\sum_{j = 1,\, j \neq i}^{l} \gamma_{ij},~i = j, \\
-\gamma_{ij},~\qquad i\neq j.
\end{cases}
\end{equation*}
Notably, $\mathcal{L}$ always has a zero eigenvalue with corresponding eigenvector $\mathbf{1}_l := \begin{bmatrix}
    1 & 1 & \cdots & 1
\end{bmatrix}^\top\in\R^l$, \emph{i.e.},
$\mathbf{1}_l \in \ker \mathcal{L}$.

A directed graph is strongly connected, if there is a directed path from every node to every other node. An undirected graph is connected, if there is an undirected path between every pair of distinct nodes. A graph is called balanced, if $\sum_{j=1}^{l}\gamma_{ij} = \sum_{j=1}^{l}\gamma_{ji}$ for all $i=1,\ldots,l$~\cite{ren2007information}. For an undirected graph, $\mathcal{A}$ is symmetric, and thus every undirected graph is balanced. 

Moreover, it follows from \cite[p. $147$]{merris1994laplacian} that, for an undirected graph, $0$ is a simple eigenvalue of $\mathcal{L}$ if, and only if, the undirected graph is connected. It follows~\cite[Prop. $3$]{fax2004information} that, for a directed graph, $0$ is a simple eigenvalue of $\mathcal{L}$, if the directed graph is strongly connected. The converse of this statement is not true in general. This motivates the following assumption.

\begin{assum}\label{assum1}
The communication graph $\mathcal{G}$ is either 
\begin{itemize}
    \item directed, balanced and strongly connected, or
    \item undirected and connected. 
\end{itemize}
\end{assum}

The next result is an immediate consequence of the above discussion together with the findings of~\cite{ren2007information} and that the Laplacian
matrix associated with any balanced graph is positive semi-definite~\cite{olfati2004consensus}.

\begin{lem}\label{lemma:Laplacian}
Consider a communication graph $\mathcal{G}$ with Laplacian matrix $\mathcal{L}$ such that  Assumption~\ref{assum1} is satisfied. Then $0$ is a simple eigenvalue of the Laplacian matrix $\mathcal{L}$ and the right and left eigenvectors associated with the eigenvalue $0$  are $\mathbf{1}_l$. Furthermore, $\mathcal{L}$ is positive semi-definite.
\end{lem} 

\section{Partial state estimation of state-space system}\label{sec:classical}

In this section we propose a novel partial state estimator design for (centralized) LTI state-space systems~\eqref{dls}.
First, we provide a decomposition for the coefficient matrices of system~\eqref{dls}, which will play a crucial role in the establishment of Theorem~\ref{thm:obsvexistence}.

\begin{lem}\label{lm:decomp}
Consider system \eqref{dls}. Then there exists an orthogonal matrix $P\in\R^{n\times n}$ such that
\begin{subequations}\label{kalmanobsvstab}
\begin{eqnarray}
& P^\top AP = \begin{bmatrix} A_{o_1} & 0 & 0 \\ A_{21} & A_{{\bar{o}}_2} & 0 \\ A_{31} & A_{32} & A_{{\bar{o}}_3} \end{bmatrix},
~P^\top B = \begin{bmatrix} B_{o_1} \\ B_{{\bar{o}}_2} \\ B_{{\bar{o}}_3} \end{bmatrix}, & \\
& CP = \begin{bmatrix} C_{o_1} & 0 & 0 \end{bmatrix}, ~ KP = \begin{bmatrix} K_{o_1} & K_{{\bar{o}}_2} & K_{{\bar{o}}_3} \end{bmatrix} ,&
\end{eqnarray}
\end{subequations}
where 
\begin{enumerate}
\item  $A_{o_1} \in \mathbb{R}^{n_1 \times n_1}$, $C_{o_1} \in \mathbb{R}^{p \times n_1}$, and the pair $(A_{o_1},C_{o_1})$ is observable,
\item  $A_{\bar{o}_2} \in \mathbb{R}^{n_2 \times n_2}$ is stable, \emph{i.e.}, $\sigma(A_{\bar{o}_2}) \subseteq \mathbb{C}^-$,
\item $A_{\bar{o}_3} \in \mathbb{R}^{n_3 \times n_3}$ is unstable, \emph{i.e.}, $\sigma(A_{\bar{o}_3}) \subseteq \oC$.
\end{enumerate}
\end{lem}

\begin{pf}
Consider matrix pair $(A,C)$ with observability matrix $$\mathcal{O}_{[A,C],n} := \begin{bmatrix}
C^\top & (CA)^\top & \ldots & (CA^{n-1})^\top \end{bmatrix}^\top.$$ 
Using Singular Value Decomposition (SVD), compute an orthogonal matrix $P_1$ such that 
\begin{equation}\label{eq:svd}
\mathcal{O}_{[A,C],n}P_1 = \begin{bmatrix}
\mathcal{W} & 0 \end{bmatrix}
\end{equation} 
and $\mathcal{W}$ has full column rank. Then  $\mathcal{W}= \mathcal{O}_{[A_{o_1},C_{o_1}],n}$, $P_1^\top AP_1 = \begin{bmatrix} A_{o_1} & 0 \\ A_{2} & A_{\bar{o}_1} \end{bmatrix}$, and $CP_1 = \begin{bmatrix} C_{o_1} & 0 \end{bmatrix}$, where $(A_{o_1},C_{o_1})$ is observable and $A_{\bar{o}_1}$ corresponds to the unobservable modes of $(A,C)$. Now, in view of Proposition \ref{lm:lower:triangular}, compute an orthogonal matrix $P_2$ such that \begin{equation}\label{eq:Schur:decomp}
P_2^\top A_{\bar{o}_1}P_2 = \begin{bmatrix}
A_{\bar{o}_2} & 0 \\ A_{32} & A_{\bar{o}_3}
\end{bmatrix} ,
\end{equation} where $\sigma(A_{\bar{o}_2}) \subseteq \mathbb{C}^- $ and $\sigma(A_{\bar{o}_3}) \subseteq \oC $. Thus, the result is shown with $P = P_1 \begin{bmatrix}
I_{n_1} & 0 \\ 0 & P_2 \end{bmatrix}$.
\end{pf}

\begin{defn}[{\cite{jaiswal2025partial}}]\label{def:kdetectability}
The system \eqref{dls}, or pair $(A,C)$, is said to be partially detectable with respect to $K$, if for all solutions $(x_1, u, y, z_1)$,  $(x_2, u, y, z_2) \in \mathscr{L}^1_{loc}(\mathbb{R};\mathbb{R}^{n+m+p+r})$ of \eqref{dls}, we have
$$ z_1(t) - z_2(t) \to 0 \text{ as } t \to \infty,$$ 
where $\mathscr{L}^1_{loc}(\mathbb{R};\mathbb{R}^{n+m+p+r})$ is the set of locally (Lebesgue) integrable $\mathbb{R}^{n+m+p+r}$-valued functions. Equivalently, for each solution $(x, 0, 0, z) \in \mathscr{L}^1_{loc}(\mathbb{R};\mathbb{R}^{n+m+p+r})$ of \eqref{dls}, we have $z(t) \to 0$ as $t \to \infty$.
\end{defn}

\begin{thm}\label{thm:Kdetectability}
For system \eqref{dls} and a decomposition~\eqref{kalmanobsvstab} (for some fixed matrix~$P$), the following statements are equivalent:
\begin{enumerate}
\item $(A,C)$ is partially detectable with respect to $K$. \label{thm:Kdetect:a}
\item The triple $(A,C,K)$ satisfies \label{thm:Kdetect:c}
\begin{equation}\label{eq:Kdetect:c}
\forall\, \lambda \in \oC:\ \rank \begin{bmatrix}
\mathcal{D}_{[A,C],n,\lambda} \\ K \end{bmatrix} = \rank \mathcal{D}_{[A,C],n,\lambda}. 
\end{equation}
\item $K_{\bar{o}_3}=0$ in \eqref{kalmanobsvstab}. \label{thm:Kdetect:b}
\end{enumerate}
\end{thm}

\begin{pf}
$(\ref{thm:Kdetect:a}) \Leftrightarrow (\ref{thm:Kdetect:c})$ follows from \cite[Sec.~$4.2$]{jaiswal2025partial}.

$(\ref{thm:Kdetect:c}) \Leftrightarrow (\ref{thm:Kdetect:b}):$ We break the proof into two steps. \\
\textbf{Step $1$:} We show by induction on~$k$ that observability of $(A_{o_1},C_{o_1})$ implies
\begin{eqnarray}\label{r4} 
\forall\, \lambda \in \mathbb{C}:\ \rank \mathcal{D}_{[A_{o_1},C_{o_1}], k,\lambda} = n_1
\end{eqnarray}
for all $k=1,\ldots,n$. Eq. \eqref{r4} clearly holds for $k=1$. Now, let $k\le n-1$ and assume that \eqref{r4} holds for all integers $ m \leq k$, i.e., for all $\lambda \in \mathbb{C}$ and all $m \in \{1,2,\ldots,k\}$ we have
\begin{equation}\label{eq:ind:hyp}
\rank \mathcal{D}_{[A_{o_1},C_{o_1}], m,\lambda} = n_1.
\end{equation} 
Observe that
\begin{eqnarray}\label{r5}
\mathcal{D}_{[A_{o_1},C_{o_1}], k+1,\lambda} = \begin{bmatrix}
I_{kp} & 0 \\ 0 & \mathcal{D}_{[A_{o_1},C_{o_1}], 1,\lambda} \end{bmatrix} \mathcal{D}_{[A_{o_1},C_{o_1}], k,\lambda} .
\end{eqnarray}
Now, applying Proposition \ref{pre:prop5} for $X = \begin{bmatrix}
I_{kp} & 0 \\ 0 & \mathcal{D}_{[A_{o_1},C_{o_1}], 1,\lambda} \end{bmatrix}$ and the induction hypothesis \eqref{eq:ind:hyp} on \eqref{r5}, we have
\begin{align*}
\rank \mathcal{D}_{[A_{o_1},C_{o_1}], k+1,\lambda} = \rank \mathcal{D}_{[A_{o_1},C_{o_1}], k,\lambda} = n_1,~~\forall~ \lambda \in \mathbb{C}. 
\end{align*}
Thus, \eqref{r4} holds for all $k=1,\ldots,n$.	\\	
\textbf{Step $2$:} In view of decomposition \eqref{kalmanobsvstab}, condition \eqref{eq:Kdetect:c} becomes
\allowdisplaybreaks
\begin{align}\label{r3}
& \rank \begin{bmatrix}
\mathcal{D}_{[A_{o_1},C_{o_1}],n,\lambda} & 0 & 0 \\
\boxtimes_1 & (\lambda I- A_{{\bar{o}}_2})^n & 0\\ \boxtimes_2 & \boxtimes_3 & (\lambda I- A_{{\bar{o}}_3})^n \\ 
K_{o_1} & K_{{\bar{o}}_2} & K_{{\bar{o}}_3} \end{bmatrix} \nonumber \\&~~   = \rank \begin{bmatrix}
\mathcal{D}_{[A_{o_1},C_{o_1}],n,\lambda} & 0 & 0 \\
\boxtimes_1 & (\lambda I- A_{{\bar{o}}_2})^n & 0 \\ \boxtimes_2 & \boxtimes_3 & (\lambda I- A_{{\bar{o}}_3})^n \end{bmatrix}
\end{align}
where $\boxtimes_1$, $\boxtimes_2$, $\boxtimes_3$, represent some matrices of no interest. Here, matrix $\mathcal{D}_{[ A_{o_1},C_{o_1}],n,\lambda}$ has full column rank by Step~$1$. Also, $\sigma(A_{{\bar{o}}_2)} \subseteq \mathbb{C}^-$ implies that $(\lambda I- A_{{\bar{o}}_2})^n$ is invertible for each $\lambda \in \oC$. Therefore, by application of Proposition~\ref{pre:prop1} for 
$X = \begin{bmatrix}
\mathcal{D}_{[A_{o_1},C_{o_1}],n,\lambda} & 0 \\ \boxtimes_1 & (\lambda I- A_{{\bar{o}}_2})^n  
\end{bmatrix}$ with full column rank, Eq. \eqref{r3} holds if, and only if,
\begin{eqnarray*}
&& \rank \begin{bmatrix} (\lambda I - A_{{\bar{o}}_3})^n \\ K_{{\bar{o}}_3} \end{bmatrix} = \rank  (\lambda I - A_{{\bar{o}}_3})^n ,~~\forall ~ \lambda \in \oC,\\
\text{i.e.}, &&~\bigcup_{\lambda \in \oC} \ker (\lambda I - A_{{\bar{o}}_3})^n  \subseteq \ker K_{{\bar{o}}_3}, \\
\text{i.e.}, && ~~\mathbb{C}^{n_3} \subseteq \ker K_{{\bar{o}}_3} ~~ \left(\text{because } \sigma(A_{{\bar{o}}_3}) \subseteq \oC \right) .
\end{eqnarray*}
Therefore, condition \eqref{eq:Kdetect:c} holds if, and only if, $K_{{\bar{o}}_3} = 0$.
\end{pf}

\begin{lem}\label{lm:part:N}
Assume that $(A,C)$ is partially detectable with respect to $K$. Then, there exist two matrices $T$ and $L$ of appropriate dimensions such that 
\begin{enumerate}
\item $N := (TA - LC)T^\top$ is stable, and 
\item $NT+LC = TA$ and $KT^\top T = K$.
    \end{enumerate}
\end{lem}

\begin{pf}
It follows from Lemma \ref{lm:decomp} and Theorem \ref{thm:Kdetectability}, invoking that $(A,C)$ is partially detectable with respect to $K$, that there exists an orthogonal matrix $P$ such that 
\begin{eqnarray*}
& P^\top AP = \begin{bmatrix} A_{o_1} & 0 & 0 \\ A_{21} & A_{{\bar{o}}_2} & 0 \\ A_{31} & A_{32} & A_{{\bar{o}}_3} \end{bmatrix},~ CP = \begin{bmatrix} C_{o_1} & 0 & 0 \end{bmatrix}, &  \\
& KP = \begin{bmatrix} K_{o_1} & K_{{\bar{o}}_2} & 0 \end{bmatrix} ,& 
\end{eqnarray*}
Choose a matrix $L_o$ such that $\sigma(A_{o_1} - L_oC_{o_1}) \subseteq \mathbb{C}^-$, which is possible by observability of $(A_{o_1},C_{o_1})$. Set $T := \begin{bmatrix}
I_{n_1+n_2} & 0 \end{bmatrix} P^\top$ and $L := \begin{bmatrix} L_o \\ 0  \end{bmatrix}$. Then,
\allowdisplaybreaks
\begin{align*}
N &= (TA-LC)T^\top \\
&= \begin{bmatrix}
I_{n_1+n_2} & 0 \end{bmatrix} P^\top AP\begin{bmatrix}
I_{n_1+n_2} \\ 0 \end{bmatrix}  - LCP \begin{bmatrix}
I_{n_1+n_2} \\ 0 \end{bmatrix}  \\
&= \begin{bmatrix}
A_{o_1} - L_oC_{o_1} & 0 \\ A_{21} & A_{{\bar{o}}_2}
\end{bmatrix} .    
\end{align*}
Since $\sigma(A_{{\bar{o}}_2}) \subseteq \mathbb{C}^-$ it follows that $N$ is stable. Moreover, 
\begin{align}\label{eq:N:T:A}
    NT+LC &= \begin{bmatrix}
A_{o_1} - L_oC_{o_1} & 0 \\ A_{21} & A_{{\bar{o}}_2}\end{bmatrix} \begin{bmatrix}
I_{n_1+n_2} & 0 \end{bmatrix} P^\top \nonumber \\
&\quad + \begin{bmatrix}
    L_o \\ 0 
\end{bmatrix} \begin{bmatrix} C_{o_1} & 0 & 0 \end{bmatrix}P^\top \nonumber \\
&= \begin{bmatrix}
A_{o_1} & 0 & 0 \\ A_{21} & A_{{\bar{o}}_2} & 0 \end{bmatrix} P^\top \nonumber \\
&= \begin{bmatrix}
I_{n_1+n_2} & 0 \end{bmatrix}
\begin{bmatrix} A_{o_1} & 0 & 0 \\ A_{21} & A_{{\bar{o}}_2} & 0 \\ A_{31} & A_{32} & A_{{\bar{o}}_3} \end{bmatrix} P^\top \nonumber \\
&= \begin{bmatrix}
I_{n_1+n_2} & 0 \end{bmatrix} P^\top A P P^\top = TA,
\end{align}
and
\begin{align*}
 KT^\top T
 &= KP\begin{bmatrix}
     I_{n_1+n_2} \\ 0 
 \end{bmatrix}\begin{bmatrix}
     I_{n_1+n_2} & 0 
 \end{bmatrix}P^\top \\
 &= \begin{bmatrix}
K_{o_1} & K_{{\bar{o}}_2} & 0 \end{bmatrix} \begin{bmatrix}
    I_{n_1+n_2} & 0 \\ 0 & 0 
\end{bmatrix}P^\top \\
&= \begin{bmatrix}
K_{o_1} & K_{{\bar{o}}_2} & 0 \end{bmatrix} P^\top = K. 
\end{align*}
This completes the proof.
\end{pf}

In the remainder of this section, we establish a systematic procedure for the design of matrices $N\in \mathbb{R}^{q\times q}$, $H \in \mathbb{R}^{q \times m}$, $L \in \mathbb{R}^{q \times p}$, and $R\in \mathbb{R}^{r\times q}$ for some $q \in \mathbb{N} \cup \{0\}$ in such a way that the following system becomes a partial state estimator for system \eqref{dls}.
\begin{subequations}\label{obs}
\begin{align}
\dot{w}(t) &= Nw(t) + Hu(t) + Ly(t), \label{obsa} \\
\hat{z}(t) &= Rw(t) , \label{obsb}
\end{align}
\end{subequations}
First, we recall the definition of partial state estimators.

\begin{defn}\label{def:observer}
System \eqref{obs} is said to be a partial state estimator for~\eqref{dls}, if for every solution $(x,u,y,z) \in \mathscr{L}^1_{loc}(\mathbb{R};\mathbb{R}^{n+m+p+r})$ of \eqref{dls} and $\left(w, \begin{bmatrix}
u \\ y \end{bmatrix} ,\hat{z} \right) \in \mathscr{L}^1_{loc}(\mathbb{R};\mathbb{R}^{q+m+p+r})$ of \eqref{obs}, we have 
$$\hat{z}(t) - z(t) \to 0\ \text{ as }\ t\to \infty.$$
\end{defn}

\begin{thm}\label{thm:obsvexistence}
There exists a partial state estimator of the form \eqref{obs} for system \eqref{dls}, if $(A,C)$ is partially detectable with respect to $K$.
\end{thm}
	
\begin{pf} 
Assume that $(A,C)$ is partially detectable with respect to $K$. Then it follows from Lemma~\ref{lm:part:N} that there exist two matrices $T$ and $L$ of appropriate dimensions such that $N := (TA - LC)T^\top$ is stable, $NT+LC = TA$ and $KT^\top T = K$. Set $R := KT^\top$ and $H := TB-LD$.
Now, we claim that the following system is a partial state estimator for system \eqref{dls}:
\allowdisplaybreaks
\begin{align*}
\dot{w}(t) &= N w(t) + Hu(t) + Ly(t),  \\
\hat{z}(t) &= Rw(t).	
\end{align*}
Let $e(t) = \hat{z}(t) - z(t)$,  then
\begin{eqnarray*}
e(t) = Rw(t) - Kx(t) = R(w(t) - Tx(t)) =: Re_1(t) ,
\end{eqnarray*}
 and
\begin{align*}
\dot{e}_1(t) &= N w(t) + Hu(t) + Ly(t) - TAx(t) - TBu(t)\\
&= Ne_1(t) + \underset{ = 0}{\underbrace{(H + LD - TB)}}u(t) + \\
& ~\qquad \underset{ = 0}{\underbrace{(NT + LC - TA)}}x(t), \\
&= Ne_1(t).
\end{align*}
Since $\sigma(N) \subseteq \mathbb{C}^-$ we find that $e_1(t) \to 0$ as $t \to \infty$, consequently, $e(t) \to 0$ as $t \to \infty$. This completes the proof.
\end{pf}

\section{Distributed partial state estimation of state-space system}\label{sec:distributed}

In this section, first, we extend the definition of partial detectability of system \eqref{dls} to joint partial detectability of system \eqref{eqn1}. Later, we prove that under Assumption \ref{assum1}, joint partial detectability is necessary and sufficient for the existence of distributed partial state estimators for system~\eqref{eqn1}. Before proceeding further, we introduce the notation $\tilde{y} := \begin{bmatrix}
 y_1^\top & y_2^\top & \ldots & y_l^\top  
\end{bmatrix}^\top$, $\tilde{C}= \begin{bmatrix}
 C_1^\top & C_2^\top & \ldots & C_l^\top
\end{bmatrix}^\top$, and     $\tilde{p} = \sum_{i = 1}^{l} p_i $.

\begin{defn}\label{def:partial:joint:detectability}
System \eqref{eqn1} is said to be  jointly partially detectable with respect to $K$, if $(A,\tilde{C})$ is partially detectable with respect to~$K$. 
\end{defn}

The following result is a direct consequence of Definition \ref{def:partial:joint:detectability} and Theorem \ref{thm:Kdetectability}.

\begin{lem}\label{lm:joint:part:detect}
System \eqref{eqn1} is jointly partially detectable with respect to $K$ if, and only if, 
$\rank \begin{bmatrix}
\mathcal{D}_{[A,\tilde{C}],n,\lambda} \\ K \end{bmatrix} = \rank \mathcal{D}_{[A,\tilde{C}],n,\lambda} $ for all $\lambda \in \oC. $
\end{lem}

For a communication graph $\mathcal{G}$ with adjacency matrix $\mathcal{A} = [\gamma_{ij}] \in \R^{l \times l}$, we consider the following local candidate systems for a distributed partial state estimator, where $i \in \{1,2,\ldots,l\}$:
\begin{subequations}\label{eq:part:obsv}
\begin{align}
\dot{w}_i(t) &= N_iw_i(t) + H_iu(t) + L_iy_i(t) \notag \\ 
&\quad + \gamma M_i\sum_{j = 1}^{l} \gamma_{ij}(w_j(t) - w_i(t)), \\
z_i(t) &= R_iw_i(t),
  \end{align}     
\end{subequations}
where $z_i :\mathbb{R} \to \mathbb{R}^r$ is the estimate of the $i$th subsystem for the functional vector $z = Kx$, and, for some $l_i \in \mathbb{N} \cup \{0\}$, we have $N_i \in \R^{l_i \times l_i}$, $H_i \in \R^{l_i \times m}$, $L_i \in \R^{l_i \times p_i}$, $M_i \in \R^{l_i \times l_i}$, $R_i \in \R^{r \times l_i}$   and $\gamma >0$. We set $\tilde{w} := \begin{bmatrix}
 w_1^\top & w_2^\top & \ldots & w_l^\top  
\end{bmatrix}^\top$, $\hat z := \begin{bmatrix}
 z_1^\top & z_2^\top & \ldots & z_l^\top  
\end{bmatrix}^\top$ and $\tilde{l} := \sum_{i=1}^l l_i$. 

We extend the definition of distributed full-state observer that achieve omniscience asymptotically from~\cite{park2017design} to distributed partial state estimators in the following.

\begin{defn}\label{def:full:observer}
A system of the form \eqref{eq:part:obsv} is said to be a distributed partial state estimator for~\eqref{eqn1}  that  achieves omniscience asymptotically, if for every solution $\left(x,u,\tilde{y},z\right) \in \mathscr{L}^1_{loc} \left(\R;\R^{n+m+\tilde{p}+r} \right)$ of~\eqref{eqn1} and $\left(\tilde{w}, \begin{bmatrix}
    u \\ \tilde{y}
\end{bmatrix}, \hat{z}\right) \in \mathscr{L}^1_{loc} \left(\R;\R^{\tilde{l} +m+\tilde{p}+lr} \right) $  of \eqref{eq:part:obsv}, we have
\begin{equation*}
z_i(t) - z(t)\to 0\ \text{ as }\ t\to\infty, \qquad \text{ for all } i \in \{1,2,\ldots,l \},  
\end{equation*}
i.e., the estimate obtained by each node asymptotically converges to the true partial state of the plant.
\end{defn}

\begin{thm}\label{thm:distributed:full}
Consider system \eqref{eqn1} and a communication graph $\mathcal{G}$ which satisfies Assumption \ref{assum1}. Then, there exists a distributed partial state estimator of the form \eqref{eq:part:obsv} for \eqref{eqn1} that achieves omniscience asymptotically if, and only if, system \eqref{eqn1} is jointly partially detectable with respect to $K$.
\end{thm}

\begin{pf} 
($\Rightarrow$): Let system \eqref{eq:part:obsv} be a distributed partial state estimator for system \eqref{eqn1} that achieves omniscience asymptotically.  Therefore, for the solution $(0,0,0,0) \in \mathscr{L}^1_{loc} \left(\R;\R^{n +m+\tilde{p}+r} \right) $ of~\eqref{eqn1} and for an arbitrary solution $\left(\tilde{w},\begin{bmatrix}
    0\\ 0
\end{bmatrix},\hat{z}\right) \in \mathscr{L}^1_{loc} \left(\R;\R^{\tilde{l} +m+\tilde{p}+lr} \right)$ of~\eqref{eq:part:obsv}, we have 
\begin{equation}\label{part:r6}
z_i(t) = z_i(t)-z(t) \to 0 \text{ as } t \to \infty, ~ \forall ~ i \in \{1,2,\ldots,l \}.
\end{equation}
Now, let $(x,0,0,z) \in \mathscr{L}^1_{loc} \left(\R;\R^{n +m+\tilde{p}+r} \right)$ and $\left(\tilde{w},\begin{bmatrix}
    0\\ 0
\end{bmatrix},\hat{z}\right) \in \mathscr{L}^1_{loc} \left(\R;\R^{\tilde{l} +m+\tilde{p}+lr} \right)$ be arbitrary solutions of \eqref{eqn1} and \eqref{eq:part:obsv}, respectively. Then, for each $i \in \{ 1,2,\ldots,l \}$, we have
\begin{eqnarray}\label{part:r7}
e_i(t) := z_i(t)-z(t) \to 0 \text{ as } t \to \infty,
\end{eqnarray}
and combining~\eqref{part:r6} with~\eqref{part:r7} leads to
\begin{eqnarray*}
 z(t) = z_i(t) - e_i(t)\to 0 \text{ as } t \to \infty ,~ \forall ~ i \in \{ 1,2,\ldots,l \}.
\end{eqnarray*}
Hence, it follows from Definition \ref{def:kdetectability} and Definition \ref{def:partial:joint:detectability} that pair $(A,\tilde{C})$ is jointly partially detectable with respect to $K$.

$(\Leftarrow):$ Assume that system \eqref{eqn1} is jointly partially detectable with respect to $K$. Then it follows from Definition \ref{def:partial:joint:detectability} and Lemma~\ref{lm:part:N} that there exist matrices $T$ and $\bar{L} = \begin{bmatrix}
    \bar{L}_1 & \bar{L}_2 & \ldots & \bar{L}_l
\end{bmatrix}$ of appropriate dimensions such that 
$\bar{N}:=  (T A - \sum_{i = 1}^{l}\bar{L}_iC_i) T^\top$ is stable, $\bar{N}T+ \bar{L}\tilde{C} = TA$ and $KT^\top T = K$.  
Now, the Lyapunov stability theory for linear systems implies there exists a symmetric positive definite matrix $\bar{P}$ such that 
\begin{equation*}
\bar{N}^\top \bar{P} + \bar{P} \bar{N} < 0 .
\end{equation*} 
For each $i \in \{ 1,2,\ldots,l\}$, set
$L_i = l\bar{L}_i$ and $N_i = (TA - L_iC_i)T^\top$.
Note that
$\sum_{i = 1}^{l}N_i =   (lTA - \sum_{i = 1}^{l}L_iC_i)T^\top = l(TA-\sum_{i = 1}^{l} \bar{L}_iC_i)T^\top = l\bar{N}$
is stable and 
\[
\left( \sum_{i = 1}^{l}N_i^\top \right) \bar{P} + \bar{P} \left( \sum_{i = 1}^{l}N_i \right)  = l \left( \bar{N}^\top \bar{P} + \bar{P}\bar{N} \right) < 0,
\]
since $l$ is a positive integer. Moreover, by performing the similar steps as for \eqref{eq:N:T:A}, we obtain
\begin{equation}\label{part:R5}
N_i T + L_i C_i = TA, \qquad \text{ for all } i \in \{ 1,2,\ldots,l\}.
\end{equation}
For each $i \in \{ 1, 2, \ldots,l \}$, define
\begin{eqnarray}\label{part:R6}
H_i = TB - L_iD_i, \quad  M_i = \bar{P}^{-1}, \quad  R_i = KT^\top,
\end{eqnarray}
and the system \eqref{eq:part:obsv}
for some positive $\gamma>0$ to be specified later. Let $e_i(t) = z_i(t) - z(t)$,  then
\begin{eqnarray}\label{part:R3}
e_i(t) = R_iw_i(t) - Kx(t) = R_i(w_i(t) - Tx(t)) =: R_i\tilde{e}_i(t) .
\end{eqnarray}
Furthermore, we find that
\begin{align}\label{part:R4}
&\dot{\tilde{e}}_i(t) = N_i\tilde{e}_i(t) +  (N_iT + L_iC_i - TA)x(t)  \nonumber \\
& \quad + (H_i + L_iD_i - TB)u(t) + \gamma M_i\sum_{j = 1}^{l} \gamma_{ij}(\tilde{e}_j(t) - \tilde{e}_i(t)).
\end{align}
Therefore, from \eqref{part:R5}, \eqref{part:R6}, \eqref{part:R3}, and \eqref{part:R4},  the error dynamics of the $i^{th}$-local estimator become 
\begin{align*}
\dot{\tilde{e}}_i (t) &= N_i\tilde{e}_i(t) + \gamma M_i\sum_{j = 1}^{l} \gamma_{ij}(\tilde{e}_j(t) - \tilde{e}_i(t)) , \\
e_i(t) &= R_i\tilde{e}_i(t) . 
\end{align*}
Consequently, by setting $N = \blkdiag\left(N_1,N_2,\ldots,N_l\right)$, $M = \blkdiag\left(M_1,M_2,\ldots,M_l\right)$, $R = \blkdiag\left(R_1,R_2,\ldots,R_l\right)$, $\tilde{e} = \begin{bmatrix}
\tilde{e}_1^\top & \tilde{e}_2^\top & \ldots & \tilde{e}_l^\top      
\end{bmatrix}^\top$ and  $e = \begin{bmatrix}
e_1^\top & e_2^\top & \ldots & e_l^\top
\end{bmatrix}^\top$, 
the global error dynamics become
\begin{align*}
\dot{\tilde{e}} (t) &= \left(N - \gamma M(\mathcal{L} \otimes I_{\tilde{l}} ) \right) \tilde{e}(t) , \\
e(t) &= R\tilde{e}(t). 
\end{align*} 
To complete the proof, it is sufficient to find the value of  $\gamma$ in such a way that $N - \gamma M(\mathcal{L} \otimes I_{\tilde{l}})$ is stable. Utilizing the symmetry and positive definiteness of $M$, we construct a Lyapunov function as
$V(t) = \tilde{e}^\top (t) M^{-1} \tilde{e}(t)$. Then
\begin{align*}
\dot{V}(t) &= \dot{\tilde{e}}^\top (t) M^{-1} \tilde{e}(t) + \tilde{e}^\top (t) M^{-1} \dot{\tilde{e}} (t) \\
&= \tilde{e}^\top (t) \left(N^\top M^{-1} + M^{-1}N - \gamma (\mathcal{L} + \mathcal{L}^\top) \otimes I_{\tilde{l}} \right) \tilde{e}(t).
\end{align*}
To complete the proof, it remains to show that $\dot{V}(t) < 0$ for all $t\ge 0$. By Assumption~\ref{assum1} and Lemma~\ref{lemma:Laplacian}, the Laplacian matrix $\mathcal{L}$ is positive semi-definite and $\mathbf{1}_l \in\ker (\mathcal{L} + \mathcal{L}^\top)$. Consequently, $\mathcal{L} + \mathcal{L}^\top$ is symmetric and positive semi-definite. We may decompose any vector $\tilde{e} = w_1 + w_2$, where $w_1 = \mathbf{1}_l \otimes w \in \ker\left( (\mathcal{L} + \mathcal{L}^\top) \otimes I_{\tilde{l}} \right)$ for some $w\in\R^{\tilde l}$, and $w_2$ is orthogonal to $w_1$. Therefore,
\begin{align*}
 \dot{V}(t) 
 &= w_1^\top(t) \left(N^\top M^{-1} + M^{-1}N \right) w_1(t) \\
 & \quad + w_2^\top(t)  \left(N^\top M^{-1} + M^{-1}N \right) w_2(t)  \\
 & \quad- w_2^\top(t)  \left(\gamma (\mathcal{L} + \mathcal{L}^\top) \otimes I_{\tilde{l}} \right) w_2(t) \\ &~~~ + 2w_2^\top(t)\left(N^\top M^{-1} + M^{-1}N \right) w_1(t) \\
&\leq  l w^\top(t)\left( \bar{N}^\top \bar{P} + \bar{P}\bar{N} \right)w(t) \\
& \quad  + (\Lambda_1 - \gamma \lambda_2 ) w_2^\top(t)w_2(t) \\
& \quad  + 2w_2^\top(t)\left(N^\top M^{-1} + M^{-1}N \right) w_1(t) \\
&\leq \Lambda_2 w_1^\top(t)w_1(t) + (\Lambda_1 - \gamma \lambda_2 ) w_2^\top(t)w_2(t) \\
& \quad + 2w_2^\top(t)\left(N^\top M^{-1} + M^{-1}N \right) w_1(t)\\
&= \begin{bmatrix}
    w_1(t) \\ w_2(t)
\end{bmatrix}^\top \begin{bmatrix}
\Lambda_2 I & \left(N^\top M^{-1} + M^{-1}N\right) \\
* & (\Lambda_1 - \gamma \lambda_2) I\end{bmatrix} \begin{bmatrix}
    w_1(t) \\ w_2(t)
\end{bmatrix}
\end{align*}
where $\Lambda_{1}$ and $\Lambda_{2} (<0)$ are the largest eigenvalues of the matrices $N^\top M^{-1} + M^{-1}N$ and $\bar{N}^\top \bar{P} + \bar{P}\bar{N}$, respectively, and $\lambda_2(>0)$ is the smallest nonzero eigenvalue of  $\mathcal{L} + \mathcal{L}^\top$. Now, it follows from considering the Schur complement that the symmetric matrix $\begin{bmatrix}
-\Lambda_2 I & -\left(N^\top M^{-1} + M^{-1}N\right) \\
* & (\gamma \lambda_2 -  \Lambda_1 ) I\end{bmatrix}$ is positive definite if, and only if, $-\Lambda_2 I > 0$  and
\[
 (\gamma \lambda_2 -  \Lambda_1 ) I + \frac{1}{\Lambda_2} \left(N^\top M^{-1} + M^{-1}N\right)^2 > 0,
\]
for which it is sufficient that $\Lambda_2< 0$  and $(\gamma \lambda_2 -  \Lambda_1 ) + \frac{\Lambda_3}{\Lambda_2}  >0$, where $\Lambda_3$ is the largest eigenvalue of the matrix $\left( N^\top M^{-1} + M^{-1}N \right)^2$. Thus, for every positive scalar $\gamma > \frac{\Lambda_1 - \frac{\Lambda_3}{\Lambda_2}}{\lambda_2}$, we have
\begin{align*}
\begin{bmatrix}
\Lambda_2 I & N^\top M^{-1} + M^{-1}N \\ * & (\Lambda_1 - \gamma \lambda_2 ) I
\end{bmatrix} < 0 \implies \dot{V}(t)<0,\ t\ge 0.
\end{align*}
Hence, it follows from Lyapunov stability theory that $N - \gamma M(\mathcal{L} \otimes I_l)$ is stable. This completes the proof.
\end{pf}
We conclude this section by summarizing the estimator design steps in the following algorithm.

\begin{breakablealgorithm}\caption{Computational steps to construct a distributed partial state estimator \eqref{eq:part:obsv} for system \eqref{eqn1} under Assumption \ref{assum1} and joint partial detectability of system \eqref{eqn1} with respect to $K$}\label{alg1}
\begin{enumerate}
\item Compute an orthogonal matrix $P$ such that the properties in Lemma \ref{lm:decomp} are satisfied; an explicit construction is given in its proof. 

\item Compute $T = \begin{bmatrix}
    I_{n_1+n_2} & 0
\end{bmatrix}P^\top$.
\item Compute a matrix $\bar{L} = \begin{bmatrix}
    \bar{L}_1 & \bar{L}_2 & \ldots & \bar{L}_l
\end{bmatrix}$ such that 
$ \bar{N}:=  (TA - \sum_{i = 1}^{l}\bar{L}_iC_i)T^\top$ is stable. 
		
\item Compute a symmetric and positive definite matrix $\bar{P}$ such that 
$\bar{N}^\top \bar{P} + \bar{P}\bar{N} < 0$.

\item Compute  the largest eigenvalue $\Lambda_{2} (<0)$ of the matrix $\bar{N}^\top \bar{P} + \bar{P}\bar{N}$.

\item For each $i \in \{ 1,2,\ldots,l\}$, set $L_i = l\bar{L}_i$,  $N_i = (TA - L_iC_i)T^\top$, $
H_i = TB - L_iD_i$, $R_i = KT^\top$,  and $M_i = \bar{P}^{-1}$.

\item Define the matrices $L = \blkdiag(L_1, L_2, \ldots, L_l)$, $ N = \blkdiag(N_1, N_2, \ldots, N_l)$, and  \\
$M = \blkdiag(M_1, M_2, \ldots, M_l)$. 

\item Compute  the largest eigenvalue $\Lambda_{1}$ of matrix the $N^\top M^{-1} + M^{-1}N$.

\item Compute  the largest eigenvalue $\Lambda_{3}$ of the matrix $\left( N^\top M^{-1} + M^{-1}N \right)^2$. 

\item Compute the smallest nonzero eigenvalue $\lambda_2(>0)$ of $\mathcal{L} + \mathcal{L}^\top$. 

\item Choose a positive scalar $\gamma > \frac{\Lambda_1 - \frac{\Lambda_3}{\Lambda_2}}{\lambda_2}$. Then $N - \gamma M(\mathcal{L} \otimes I_{\tilde{l}} ) $ is stable and~\eqref{eq:part:obsv} is  a distributed partial state estimator.
\end{enumerate}
\end{breakablealgorithm}

\begin{rem}
    To reduce computational burden, Step 8) of Algorithm~\ref{alg1} can be skipped and the estimate $\Lambda_1\le \sqrt{\Lambda_3}$ (which is not sharp in general) can be used.
\end{rem}

\section{Numerical Illustration}\label{sec:numerical}

In this section, we provide a numerical example to validate the developed theoretical results.

\begin{exmp}\label{exp1}
Consider system \eqref{eqn1} with coefficient matrices
\begin{align*}
&    A = \begin{bmatrix}
-1 & 0 & 0 & 1 & 0 & 0 \\ 0 & 3 & 1 & 1 & 1 & 0 \\ 0 & 0 & 2 & 0 & 1 & 0 \\ 0 & 0 & 0 & 3 & 0 & 0 \\ 0 & 0 & 0 & 0 & -2 & 1 \\ 0 & 0 & 0 & 0 & 3 & -2
    \end{bmatrix}, ~ 
B = \begin{bmatrix}
0 & 0 \\ 0 & 0 \\ 0 & 0 \\ 2 & 1 \\ 1 & 0 \\ 4 & 1
    \end{bmatrix}, \\
& C_1 = \begin{bmatrix}
    1 &0 &0 &0 &0 & 0
\end{bmatrix},~ C_2 = \begin{bmatrix}
     0 & 0 & 1 &0 &0 & 0
\end{bmatrix}, \\
& D_1 = 0 = D_2,~
K = \begin{bmatrix}
    0 & 0 & 1 & 2 & 1 & 0 
\end{bmatrix}.
\end{align*}
The communication graph is undirected simply consists of two nodes connected by one edge. Therefore, Assumption \ref{assum1} is satisfied and the Laplacian matrix is given by $\mathcal{L} =\begin{bmatrix}
1 & -1 \\ -1 & 1
\end{bmatrix}$.
Here, it is easy to compute that $\sigma(A) =: \{ 3,~2,~ -0.2679,~ -3.7321,~ -1\}$. Now, it follows from rank condition~\eqref{eq:Kdetect:c} that  neither $(A,C_1)$ nor $(A,C_2)$ is partially detectable with respect to~$K$, as
$    \rank \begin{bmatrix}
\mathcal{D}_{[A,C_1],n,2} \\ K \end{bmatrix} = 6 \neq 5 =\rank \mathcal{D}_{[A,C_1],n,2}$
 and $\rank \begin{bmatrix}
\mathcal{D}_{[A,C_2],n,3} \\ K \end{bmatrix} = 5 \neq 4 =\rank \mathcal{D}_{[A,C_2],n,3}$.
However, the given system is jointly partially detectable with respect to~$K$ as (cf. Lemma \ref{lm:joint:part:detect})
 \begin{equation*}
 \forall\, \lambda \in \oC: \quad \rank \begin{bmatrix}
\mathcal{D}_{[A,\tilde{C}],n,\lambda} \\ K \end{bmatrix} = \rank \mathcal{D}_{[A,\tilde{C}],n,\lambda}. 
\end{equation*}
Hence, it follows from Theorem \ref{thm:distributed:full} that there exists a distributed partial state estimator of
the form \eqref{eq:part:obsv} that achieves omniscience asymptotically. We now design the same for the given system by utilizing the steps provided in Algorithm \ref{alg1}.
By following Steps~$1$ and~$2$ of Algorithm~\ref{alg1}, we obtain
\begin{align*} 
T = \begin{bmatrix}
0 & 0 & 0.3534 & 0 & 0.8463  & -0.3986 \\-0.0113 & 0 & 0 & 0.9999 & 0 & 0 \\ 0 & 0 & 0.9104 & 0 & -0.2132 & 0.3545 \\ 0.9999 & 0 & 0 & 0.0113 & 0 & 0 \\0 & 0 & 0.2150 & 0 & -0.4882 & -0.8458
\end{bmatrix}.
\end{align*}
Now, by leveraging detectability of $(TAT^\top, \tilde{C}T^\top)$, we obtain matrices
$\bar{L}_1 = \begin{bmatrix}
   0  \\ 43.6992  \\  0 \\  7.4048 \\ 0  
\end{bmatrix}$ and $\bar{L}_2 = \begin{bmatrix}
1.4284 \\ 0 \\ 2.5993 \\  0 \\ 0.2939
\end{bmatrix}$ such that matrix
\begin{align*}
   \bar{N} &= (TA-\bar{L}_1C_1-\bar{L}_2C_2)T^\top \\
   & = \begin{bmatrix}
 -3.0556  &  0   & 0.4662 &  0  & -0.3077 \\ 0 &   3.4808  &  0 & -43.6515  &  0 \\ 2.1238  & 0 &  -1.5474  &  0 &  -0.5592 \\ 0  &  1.1284  &  0 &  -8.3926  &  0 \\ -1.5709  & 0  &  0.8376   & 0  & -0.3315
\end{bmatrix} 
\end{align*} is stable. Further, in view of the Lyapunov stability theory for linear systems, we obtain the  matrix 
$$\bar{P} = \begin{bmatrix}
0.4615  & 0&    0.1435 &  0 &  -0.3568 \\0  &  0.0482   & 0 &  -0.1945  &  0 \\0.1435   & 0   & 0.5391 &  0  &  0.2515 \\0  & -0.1945 &  0  &  1.0710 &   0 \\
-0.3568  &  0  &  0.2515 &   0 &   1.0366
\end{bmatrix}$$ such that $\bar{N}^\top \bar{P} + \bar{P}\bar{N} < 0$ with $\sigma(\bar{N}^\top \bar{P} + \bar{P}\bar{N}) = \{ -1.1605,~-1.1041,~-1.0038,~ -0.6874,~ -0.0994\}$ and $\Lambda_2 = -0.0994$. Now, by following the remaining steps of Algorithm \ref{alg1}, we obtain the distributed partial state estimators as follows: 
\begin{align*}
\dot{w}_1(t) &= N_1 w_1(t) + H_1u(t) + L_1y_1(t) + \gamma M_1 (w_2(t) - w_1(t)), \\ 
z_1(t) &= R_1w_1(t),	
\end{align*}
and 
\begin{align*}
\dot{w}_2(t) &= N_2 w_2(t) + H_2u(t) + L_2y_2(t) + \gamma M_2 (w_1(t) - w_2(t)), \\ 
z_2(t) &= R_2w_2(t),	
\end{align*}
with $\gamma = 66$ and the coefficient matrices 
\begin{align*}
& N_1 = \begin{bmatrix}
-2.5508 &   0 &   1.7666  & 0 &  -0.0006 \\ 0 & 3.9733 & 0 & -87.3479 & 0 \\ 3.0424 &  0  &  0.8191&    0 &  -0.0003 \\ 0  &  1.2119 & 0 & -15.7969 & 0\\
-1.4670 & 0 & 1.1051 & 0 & -0.2683
\end{bmatrix},\\
& N_2 = \begin{bmatrix}
-3.5603  & 0 & -0.8343 & 0 & -0.6148 \\ 0 & 2.9882 & 0 & 0.0450 & 0 \\ 1.2052  & 0 & -3.9139 & 0 & -1.1180 \\ 0 & 1.0450  & 0 & -0.9882 & 0 \\ -1.6747 & 0 & 0.5700 & 0 & -0.3947
\end{bmatrix},\\
& L_1 = \begin{bmatrix}
0 \\ 87.3984 \\ 0 \\14.8096 \\ 0
\end{bmatrix} , ~L_2 = \begin{bmatrix}
2.8567 \\ 0 \\ 5.1987 \\ 0 \\ 0.5877
\end{bmatrix},\\
& R_1 = \begin{bmatrix}
1.1997 & 1.9999 & 0.6972 &   0.0225 & -0.2732
\end{bmatrix} = R_2,  \\
& M_1 = \begin{bmatrix}
4.3866   & 0 &  -2.1105  &  0  &  2.0220 \\ 0 & 77.8061 & 0 & 14.1274 & 0 \\ -2.1105 & 0 & 3.1069 & 0 &  -1.4802 \\ 0 & 14.1274  & 0  &  3.4988 &  0 \\ 2.0220  & 0  & -1.4802 &  0 &   2.0198  
\end{bmatrix} = M_2,\\
& H_1 = 
 \begin{bmatrix}
-0.7482 &  1.9999 & 1.2047 & 0.0225 & -3.8716 \\ -0.3986  & 0.9999 & 0.3545  & 0.0113  & -0.8458
\end{bmatrix}^\top = H_2 .
\end{align*}

\noindent Taking the initial conditions of system \eqref{eqn1} as $x(0) = \begin{bmatrix}
1 & 2 & 1 & 2 & 3 \end{bmatrix}^\top$ and of the distributed partial state estimator~\eqref{eq:part:obsv} as $w_1(0) = \begin{bmatrix}
2 & 4 & 4 & 6 & 5  \end{bmatrix}^\top$, $w_2(0) = \begin{bmatrix}
4 & 4 & 6 & 8 & 4 \end{bmatrix}^\top$ with $u(t) = \begin{bmatrix}
    \exp(t) & \sin(t) 
\end{bmatrix}^\top$, the difference between the true and estimated functional vector trajectories are shown in Figure~\ref{fig:fig1}.
The simulation work and all the matrix computations have been executed in MATLAB R2024b environment with relative error tolerance $10^{-8}$. From the error plots (Figure \ref{fig:fig1}), it is clear that the designed distributed partial state estimator works well for the given system.
\begin{figure}
\centering
\subfigure[Partial state estimation error $e_1(t) = z_1(t) - z(t)$.]{\includegraphics[width=.9\linewidth]{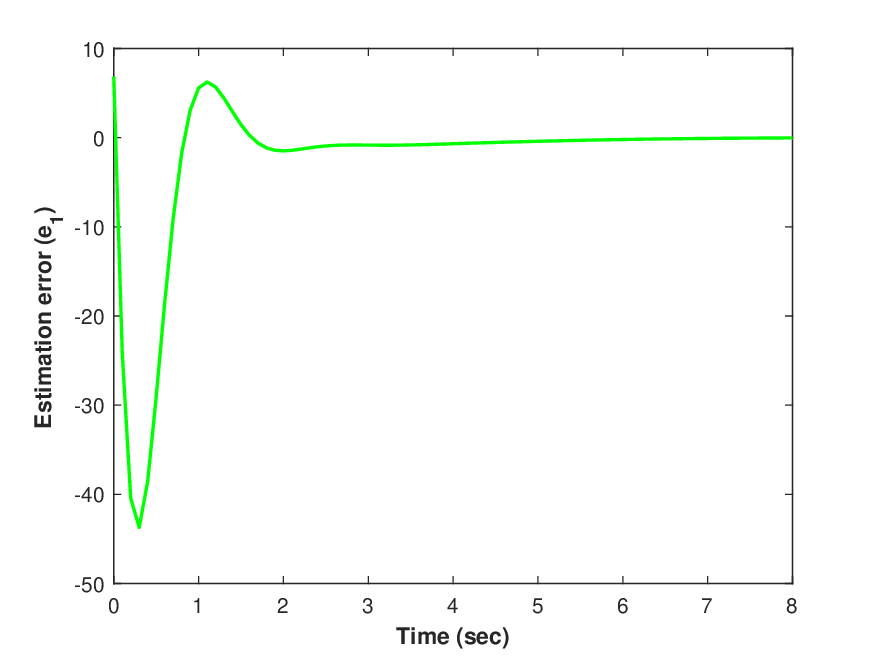}}
\subfigure[Partial state estimation error $e_2(t) = z_2(t) - z(t)$.]{\includegraphics[width=.9\linewidth]{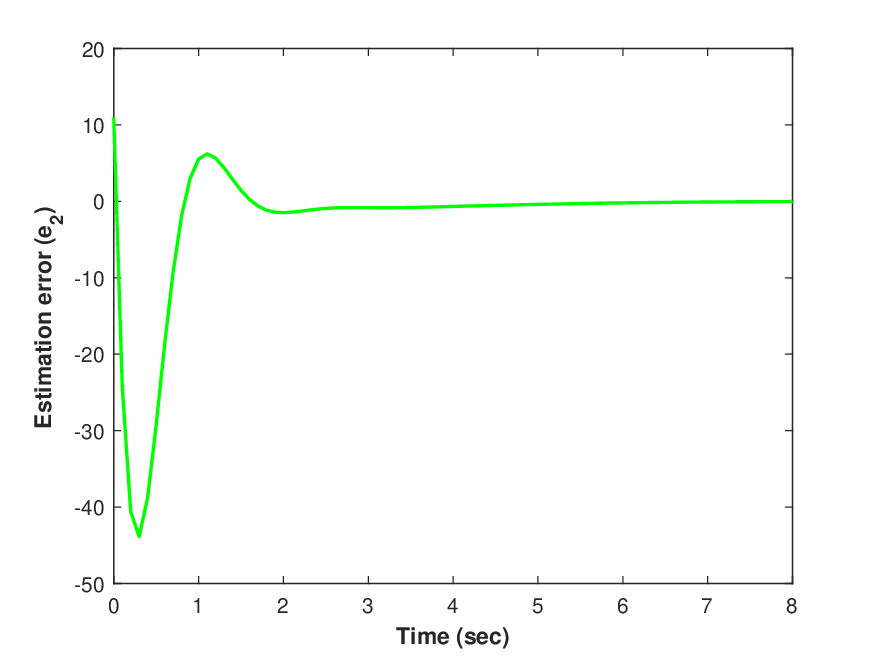}}
\caption{Time response of estimation error for the local observers}
\label{fig:fig1}
\end{figure}
\end{exmp}

\section{Conclusion}\label{sec:conclusion}
This study advances the partial state estimation theory of LTI distributed state-space systems, in which system outputs were measured via a network of sensors spread across multiple nodes. In the considered framework, local measurements at each node are insufficient for effective partial state estimation. To overcome this challenge, we have utilized the distributed state estimator framework that consists of several local estimators, each co-located with one of the given nodes and interconnected through a communication network. We  established  necessary and sufficient conditions for the existence of a distributed partial state estimator that achieves omniscience asymptotically, provided the communication graph is either directed, balanced and strongly
connected or undirected and connected. Additionally,  we proposed a novel partial state estimator design for (centralized) LTI state-space system.

\vspace{-5mm}
\bibliographystyle{plain}

\bibliography{distributed}

@book{varga2011gervsgorin,
  title={Ger{\v{s}}gorin and his circles},
  author={Varga, Richard S},
  volume={36},
  year={2011},
  publisher={Springer Science \& Business Media}
}

@inproceedings{mitra2016approach,
  title={An approach for distributed state estimation of {LTI} systems},
  author={Mitra, Aritra and Sundaram, Shreyas},
  booktitle={2016 54th Annual Allerton Conference on Communication, Control, and Computing (Allerton)},
  pages={1088--1093},
  year={2016},
  organization={IEEE}
}

@article{han2018simple,
  title={A simple approach to distributed observer design for linear systems},
  author={Han, Weixin and Trentelman, Harry L and Wang, Zhenhua and Shen, Yi},
  journal={IEEE Transactions on Automatic Control},
  volume={64},
  number={1},
  pages={329--336},
  year={2018},
  publisher={IEEE}
}

@article{park2017design,
  title={Design of distributed {LTI} observers for state omniscience},
  author={Park, Shinkyu and Martins, Nuno C},
  journal={IEEE Transactions on Automatic Control},
  volume={62},
  number={2},
  pages={561--576},
  year={2017},
  publisher={IEEE}
}

@article{ugrinovskii2013conditions,
  title={Conditions for detectability in distributed consensus-based observer networks},
  author={Ugrinovskii, Valery},
  journal={IEEE Transactions on Automatic Control},
  volume={58},
  number={10},
  pages={2659--2664},
  year={2013},
  publisher={IEEE}
}

@inproceedings{park2012augmented,
  title={An augmented observer for the distributed estimation problem for {LTI} systems},
  author={Park, Shinkyu and Martins, Nuno C},
  booktitle={2012 American Control Conference (ACC)},
  pages={6775--6780},
  year={2012},
  organization={IEEE}
}

@inproceedings{park2012necessary,
  title={Necessary and sufficient conditions for the stabilizability of a class of {LTI} distributed observers},
  author={Park, Shinkyu and Martins, Nuno C},
  booktitle={2012 IEEE 51st IEEE Conference on Decision and Control (CDC)},
  pages={7431--7436},
  year={2012},
  organization={IEEE}
}

@article{yang2022state,
  title={State estimation using a network of distributed observers with unknown inputs},
  author={Yang, Guitao and Barboni, Angelo and Rezaee, Hamed and Parisini, Thomas},
  journal={Automatica},
  volume={146},
  pages={110631},
  year={2022},
  publisher={Elsevier}
}

@article{doostmohammadian2013genericity,
  title={On the genericity properties in distributed estimation: Topology design and sensor placement},
  author={Doostmohammadian, Mohammadreza and Khan, Usman A},
  journal={IEEE Journal of Selected Topics in Signal Processing},
  volume={7},
  number={2},
  pages={195--204},
  year={2013},
  publisher={IEEE}
}

@book{piziak2007matrix,
  title={Matrix theory: from generalized inverses to Jordan form},
  author={Piziak, Robert and Odell, Patrick L},
  year={2007},
  publisher={Chapman and Hall/CRC}
}

@book{hardy2019matrix,
  title={Matrix calculus, Kronecker product and tensor product: a practical approach to linear algebra, multilinear algebra and tensor calculus with software implementations},
  author={Hardy, Yorick and Steeb, Willi-Hans},
  year={2019},
  publisher={World Scientific}
}

@inproceedings{kim2016distributed,
  title={Distributed {L}uenberger observer design},
  author={Kim, Taekyoo and Shim, Hyungbo and Cho, Dongil Dan},
  booktitle={2016 IEEE 55th Conference on Decision and Control (CDC)},
  pages={6928--6933},
  year={2016},
  organization={IEEE}
}

@inproceedings{zhu2014cooperative,
  title={On the cooperative observability of a continuous-time linear system on an undirected network},
  author={Zhu, Henghui and Liu, Kexin and L{\"u}, Jinhu and Lin, Zongli and Chen, Yao},
  booktitle={2014 International Joint Conference on Neural Networks (IJCNN)},
  pages={2940--2944},
  year={2014},
  organization={IEEE}
}

@article{cao2023distributed,
  title={Distributed unknown input observer},
  author={Cao, Ganghui and Wang, Jinzhi},
  journal={IEEE Transactions on Automatic Control},
  volume={68},
  number={12},
  pages={8244--8251},
  year={2023},
  publisher={IEEE}
}

@article{liang2024distributed,
  title={Distributed state and fault estimation for cyber-physical systems under DoS attacks},
  author={Liang, Limei and Su, Rong and Xu, Haotian},
  journal={IEEE/CAA Journal of Automatica Sinica},
  year={2024},
  publisher={IEEE}
}

@article{zhou2025adaptive,
  title={Adaptive distributed unknown input observer for linear systems},
  author={Zhou, Dan-Dan and Zhao, Ran},
  journal={Applied Mathematics and Computation},
  volume={486},
  pages={129027},
  year={2025},
  publisher={Elsevier}
}

@article{jaiswal2025partial,
  title={Partial detectability and generalized functional observer design for linear descriptor systems},
  author={Jaiswal, Juhi and Berger, Thomas and Tomar, Nutan Kumar},
  journal={Franklin Open},
  volume={10},
  pages={100238},
  year={2025},
  publisher={Elsevier}
}

@article{luenberger1964observing,
  title={Observing the state of a linear system},
  author={Luenberger, David G},
  journal={IEEE transactions on military electronics},
  volume={8},
  number={2},
  pages={74--80},
  year={1964},
  publisher={IEEE}
}

@article{darouach2025functional,
  title={On functional observability and functional observer design},
  author={Darouach, Mohamed and Fernando, Tyrone},
  journal={Automatica},
  volume={173},
  pages={112115},
  year={2025},
  publisher={Elsevier}
}

@article{merris1994laplacian,
  title={Laplacian matrices of graphs: a survey},
  author={Merris, Russell},
  journal={Linear algebra and its applications},
  volume={197},
  pages={143--176},
  year={1994},
  publisher={Elsevier}
}

@article{fax2004information,
  title={Information flow and cooperative control of vehicle formations},
  author={Fax, J Alexander and Murray, Richard M},
  journal={IEEE transactions on automatic control},
  volume={49},
  number={9},
  pages={1465--1476},
  year={2004},
  publisher={IEEE}
}

@article{ren2007information,
  title={Information consensus in multivehicle cooperative control},
  author={Ren, Wei and Beard, Randal W and Atkins, Ella M},
  journal={IEEE Control systems magazine},
  volume={27},
  number={2},
  pages={71--82},
  year={2007},
  publisher={IEEE}
}

@article{olfati2004consensus,
  title={Consensus problems in networks of agents with switching topology and time-delays},
  author={Olfati-Saber, Reza and Murray, Richard M},
  journal={IEEE Transactions on automatic control},
  volume={49},
  number={9},
  pages={1520--1533},
  year={2004},
  publisher={IEEE}
}

@article{matsaglia1974equalities,
  title={Equalities and inequalities for ranks of matrices},
  author={Matsaglia, George and PH Styan, George},
  journal={Linear and multilinear Algebra},
  volume={2},
  number={3},
  pages={269--292},
  year={1974},
  publisher={Taylor \& Francis}
}

\end{document}